\documentclass[11pt]{preprint}

\usepackage{hyperref}

\makeatletter
\AtBeginDocument{%
  \def\doi#1{\url{https://doi.org/#1}}}
\makeatother

\usepackage[full]{textcomp}
\usepackage[osf]{newtxtext}
\usepackage{comment}

\usepackage{amssymb}
\usepackage{mathtools}

\usepackage{breakurl}
\usepackage{mhenvs}
\usepackage{mhequ}
\usepackage{mhsymb}
\usepackage{booktabs}
\usepackage{tikz}
\usepackage{mathrsfs}
\usepackage{longtable}
\usepackage{graphicx}
\usepackage{subfig}
\usetikzlibrary{plotmarks}
\usetikzlibrary{decorations.pathreplacing}
\usetikzlibrary{decorations.pathmorphing}
\usetikzlibrary{fit}
\usetikzlibrary{positioning,shapes}

\usepackage{float}
\usepackage{microtype}
\usepackage{comment}
\usepackage{wasysym}
\usepackage{centernot}
\usepackage{enumitem}
\usepackage{bm}
\usepackage{stackrel}

\DeclareMathAlphabet{\mathbbm}{U}{bbm}{m}{n}

\DeclareFontFamily{U}{BOONDOX-calo}{\skewchar\font=45 }
\DeclareFontShape{U}{BOONDOX-calo}{m}{n}{
  <-> s*[1.05] BOONDOX-r-calo}{}
\DeclareFontShape{U}{BOONDOX-calo}{b}{n}{
  <-> s*[1.05] BOONDOX-b-calo}{}
\DeclareMathAlphabet{\mcb}{U}{BOONDOX-calo}{m}{n}
\SetMathAlphabet{\mcb}{bold}{U}{BOONDOX-calo}{b}{n}

\setlist{noitemsep,topsep=4pt}

\newcommand{\Ad}{\mathrm{Ad}}

\newcommand{\G}{\mathbb{G}}

\def\${|\!|\!|}

\def\id{\mathrm{id}}

\def\scal#1{{\langle#1\rangle}}

\newcommand{\mfg}{\mathfrak{g}}

\tikzset{
dot/.style={circle,fill=black,inner sep=0pt, minimum size=1.2mm}
}

\makeatletter
\DeclareRobustCommand{\cev}[1]{%
  {\mathpalette\do@cev{#1}}%
}
\newcommand{\do@cev}[2]{%
  \vbox{\offinterlineskip
    \sbox\z@{$\m@th#1 x$}%
    \ialign{##\cr
      \hidewidth\reflectbox{$\m@th#1\vec{}\mkern4mu$}\hidewidth\cr
      \noalign{\kern-\ht\z@}
      $\m@th#1#2$\cr
    }%
  }%
}
\makeatother

\newcommand{\mrd}{\mathop{}\!\mathrm{d}}

\DeclareMathOperator{\Trace}{Tr}

\usetikzlibrary{shapes.misc}
\usetikzlibrary{shapes.symbols}
\usetikzlibrary{shapes.geometric}
\usetikzlibrary{decorations}
\usetikzlibrary{decorations.markings}
\usetikzlibrary{arrows.meta}

\usetikzlibrary{calc}

\colorlet{darkblue}{blue!90!black}
\colorlet{darkred}{red!90!black}
\colorlet{darkgreen}{green!70!black}

\newcommand{\e}{\varepsilon}

\def\${|\!|\!|}

\def\id{\mathrm{id}}

\def\id{\mathrm{id}}

\def\E{\mathbb{E}}
\def\P{\mathbb{P}}
\def\R{\mathbb{R}}

\def\Z{\mathbb{Z}}

\newtheorem{assumption}[lemma]{Assumption}

\title{Villain action in lattice gauge theory}
\author{Ilya~Chevyrev$^1$ and Christophe Garban$^2$}

\institute{Institut f\"ur Mathematik, TU Berlin,
10623 Berlin, Germany, and School of Mathematics, University of Edinburgh, EH9 3FD Edinburgh, UK. \email{ichevyrev@gmail.com} \and
Universit\'e Claude Bernard Lyon 1, CNRS UMR 5208, Institut Camille Jordan, 69622 Villeurbanne, France, and Institut Universitaire de France (IUF).
 \email{christophe.garban@gmail.com}}

\date{}
\begin{document}
\maketitle
\begin{abstract}

We prove that  Villain interaction applied to lattice gauge theory can be obtained as the limit of both Wilson and Manton interactions on a  larger graph which we call the {\em carpet graph.} This is the lattice gauge theory analog of a well-known property for spin $O(N)$ models where Villain type interactions are the limit of $\mathbb{S}^{N-1}$ spin systems defined on a {\em cable graph}.  

Perhaps surprisingly in the setting of lattice gauge theory, our proof also applies to non-Abelian lattice theory such as $SU(3)$-lattice gauge theory and its limiting Villain interaction.   

In the particular case of an Abelian lattice gauge theory, this allows us to extend the validity of Ginibre inequality to the case of the Villain interaction.

\end{abstract}
\setcounter{tocdepth}{1}

\tableofcontents

\section{Introduction}\label{sec: intro}

\subsection{Context} 
Our main result (Theorem~\ref{thm:carpet} below) states that the Villain action in lattice gauge theory, for any compact connected structure group $G$, is the limit of any other action on so-called carpet graphs that we introduce, provided this action satisfies reasonable assumptions. See Figure \ref{f.Carpet} for an illustration of the carpet graph and Section \ref{sec:assump} for the precise assumptions on the class of actions we consider.
Standard actions such as Wilson and Manton (and trivially Villain) satisfy our assumptions.

This result can be seen as analogous to the fact that the 2D Villain model is the 1D scaling limit of the $XY$ model, see~\cite[Appendix~A]{AHPS21}.

The Villain action has played a prominent role in the study of spin-$O(N)$ models on $\Z^d$ since the 70's. For example the derivation by Berezinskii of the Berezinskii-Kosterlitz-Thouless phase transition relied on the special duality properties of the Villain interaction (see \cite{berezinskii1971destruction, jose1977renormalization}).  See also the recent works \cite{AHPS21, lammers2023bijecting, garban2023quantitative} as well as \cite{dubedat2022random} for a random-cluster perspective on the Villain spin model. 

The importance of Villain interaction in $O(N)$ spin systems is two-fold: $i)$ first it provides a natural extension of these spin-systems to the so-called {\em cable graph} allowing for a more powerful use of Markov's property (this idea was popularized in the context of the Gaussian free field on $\Z^d$ by Lupu in \cite{lupu2016loop}).  And $ii)$, when the spins are $\mathbb{S}^1$-valued, the duality with integer-valued height functions is particularly elegant as it involves discrete Gaussians. 

This paper focuses instead on lattice gauge theory where the basic randomness is now sampled along (oriented) edges.  
Lattice gauge theory was initially introduced by Kenneth Wilson as a particularly convenient way to introduce an ultra-violet cut-off
(the small mesh lattice)
in Yang-Mills Euclidean field theory whose construction in $\R^4$ remains an outstanding open problem. The case 
where the gauge group  $G$ is $U(1)$ corresponds to quantum electrodynamics while $G=SU(3)$ corresponds to 
the strong force. The phenomenology for lattice gauge theory with non-Abelian gauge group $SU(3)$  is 
believed to be radically different  from the Abelian case of $G=U(1)$. On the lattice $\Z^4$, for $G=SU(3)$, 
{\em confinement} or {\em area law} is expected at all $\beta$ while for $G=U(1)$  {\em nonconfinement} or 
{\em perimeter law} at large $\beta$ is proved in \cite{guth1980existence,frohlich1982massless}. 
We refer the reader to \cite{seiler1982gauge, frohlich1982massless, chatterjee2019yang,chatterjee2020wilson,garban2023improved,forsstrom2023wilson,cao2023random} and references therein for background on lattice gauge theory.

 It turns out that the Villain interaction has also long been considered in this setting of lattice gauge theory, for example
in the work of Migdal~\cite{Migdal75} on the integrability of 2D quantum Yang--Mills theory (see also~\cite{Driver89,Witten91,Levy03})
as well as in
 the seminal work \cite{frohlich1982massless}. It also plays a key role in the proof of confinement at all temperatures in the breakthrough work by G\"opfert-Mack \cite{gopfert1982proof} where Villain lattice gauge theory on $\Z^3$ is in duality with the integer-valued Gaussian free field $\Psi \colon \Z^3 \to \Z$. This duality between Villain lattice gauge theory and $\Z$-valued $k$-forms is also used in \cite{garban2023improved}.
 The Villain action was also recently used in~\cite{CC22} in the study of ultraviolet stability of the Abelian lattice Higgs model. 

As far as we know, the fact that such a Villain interaction appears as the limit of a Wilson lattice gauge theory along a ``cable-type'' graph has not been made explicit in the literature. This is the purpose of this paper where the limit goes through a plaquette version of the {\em cable graph} which we call the {\em carpet graph}. See Figure \ref{f.Carpet}. As we shall see below and perhaps somewhat surprisingly, such an extension is also valid for non-Abelian lattice gauge theories. 

In the special case of Abelian symmetry, one consequence of this work is the validity of Ginibre correlation inequality for lattice gauge theory with Villain interaction which is obtained through a limiting procedure as in~\cite[Appendix~A]{AHPS21}. See Corollary \ref{c.monotone}  below. It is unclear to us how to deduce such a correlation inequality without going through the carpet graph limit, see Remark \ref{r.monotone}. 

This paper is self-contained modulo several analytic results that we reference.

\subsection{Preliminaries}

Let $\Lambda^d = \{x\in \Z^d \,:\, |x|_\infty \leq L\}$ be a box in $\Z^d$ with side length $2L$, where $|x|_\infty = \max_{i=1,\ldots,d}|x_i|$.
Let $e_1,\ldots, e_d$ be the canonical basis of $\R^d$.
Let
\begin{equ}
E = \{(x,e_j) \,:\, x\in\Lambda^d, x+e_j \in \Lambda^d, 1\leq j\leq d\}
\end{equ}
denote the positively oriented edges that are contained in $\Lambda^d$.
Let
\begin{equ}
P=\{(x,e_i,e_j)\,:\, x,x+e_i,x+e_j\in\Lambda, 1\leq i<j\leq d\}
\end{equ}
denote the set of plaquettes of $\Lambda^d$.

Let $G$ be a compact matrix group with a Haar measure $\mrd x$.
Consider a function $Q\colon G\to [0,\infty)$ such that $Q(xyx^{-1}) = Q(y)$ and $Q(x)=Q(x^{-1})$
and $\int_G Q(x)\mrd x=1$ (one should think that $Q=e^{-S}$ for an action $S$ on $G$).
Define the probability measure $\mu_Q$ on $G^E$ by
\begin{equ}[eq:LGT]
\mu_Q(\mrd U) = Z_Q^{-1} \prod_{p\in P}Q(U_p) \mrd U
\end{equ}
where $Z_Q$ is a normalisation constant and where 
\begin{equ}
U_p \eqdef U_{(x,e_i)}U_{(x+e_i,e_j)}U_{(x+e_j,e_i)}^{-1}U_{(x,e_j)}^{-1}
\end{equ}
is the holonomy of $U$ around $p=(x,e_i,e_j)\in P$,
and $\mrd U$ is the Haar measure on $G^E$.

\subsection{Main result}

\begin{definition}[Carpet graph]
For $N\geq 1$,
let $\Lambda^d_N$ denote the lattice $\Lambda^d$ where we tile every plaquette $p\in P$ with $N^2$ plaquettes of size $\eps \times \eps$ where $\eps\eqdef N^{-1}$.
Let $E_N$ denote the positively oriented edges $(x,\eps e_j)$ of $\Lambda^d_N$ and $P_N$ its set of plaquettes $(x,\eps e_i,\eps e_j)$.
This is what we call the carpet graph (each plaquette $p\in P$ becomes a carpet of microscopic plaquettes). See Figure \ref{f.Carpet}. 
\end{definition}

\begin{figure}[!htp]
\begin{center}
\includegraphics[width=0.9\textwidth]{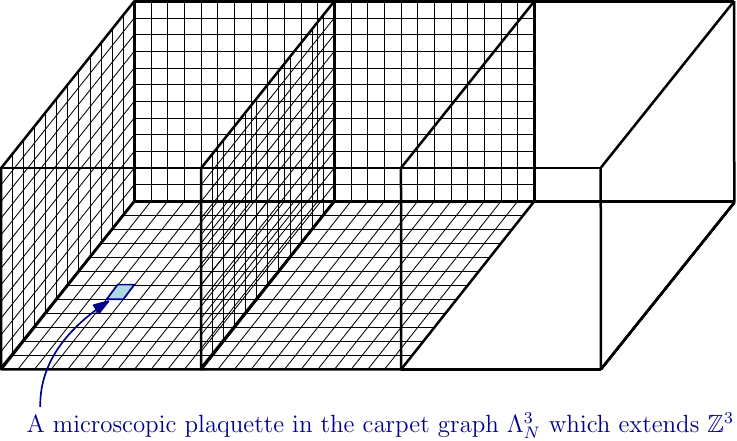}
\end{center}
\caption{An example of a carpet graph}\label{f.Carpet}
\end{figure}

For $Q\colon G\to [0,\infty)$, let $\mu_{Q,N}$ be the probability measure on $G^{E_N}$ defined by
\begin{equ}[eq:LGTN]
\mu_{Q,N}(\mrd U) = Z_{Q,N}^{-1} \prod_{p\in P_N}Q(U_p) \mrd U\;.
\end{equ}
The difference between~\eqref{eq:LGT} and~\eqref{eq:LGTN} is that the product in the latter is over $p\in P_N$.

\begin{remark}\label{rem:projection}
Every $U\in G^{E_N}$ canonically defines an element of $U\in G^E$
by taking ordered products along edges.
We denote this projection by
\begin{equ}
\pi_N\colon G^{E_N} \to G^E\;.
\end{equ}
\end{remark}

\begin{definition}[Villain action]
For $\beta>0$, define $V_\beta\colon G\to (0,\infty)$ by \begin{equ}
V_\beta(x) = e^{\frac12\beta^{-1}\Delta}(x)\;,
\end{equ}
where $\Delta = \sum_{a=1}^D T^a T^a$ is the Laplace--Beltrami operator on $G$
and $T^1,\ldots,T^D$ is a basis of the Lie algebra of $G$ as in Section \ref{sec:assump} viewed as left-invariant vector fields.
\end{definition}

\begin{example}[Abelian Villain action]
When $G=U(1) = \mathbb{S}^1$, then the Villain interaction reads as follows: for any $\theta \in [0, 2\pi)$, 
\begin{align*}
V_\beta(\theta) = \frac 1 Z \sum_{m\in \Z} \exp(-\frac \beta 2 (\theta + 2\pi m)^2)\,.
\end{align*}
\end{example}

The following is our main result.
\begin{theorem}\label{thm:carpet}
Suppose $p_N$ satisfies Assumption \ref{assump:actions} below. Then
\begin{equ}
(\pi_N)_* \mu_{p_N,N} \to \mu_{V_\beta}
\end{equ}
in total variant distance,
where $(\pi_N)_*$ denotes the pushforward along $\pi_N$.
\end{theorem}

In the special case of Abelian symmetry, $G=U(1)$, let us introduce a lattice gauge theory with edge-dependent Villain interactions on $\Lambda^d = \{x\in \Z^d \,:\, |x|_\infty \leq L\}$. Let $\bar \beta = (\beta_p)_{p\in P}$ be a field of non-negative  coupling constants assigned to plaquettes. 
We define the Villain lattice gauge theory with coupling constants $\bar \beta$ by
\begin{equ}[eq:Villain_betas]
\mu_{V, \bar \beta}(\mrd U) = Z_{V, \bar \beta}^{-1} \prod_{p\in P(\Lambda^d)}  e^{\frac 1 2 \beta_p^{-1} \Delta}(U_p) \mrd U\;.
\end{equ}

We canonically identify $E$ with a set of oriented edges connecting neighbouring vertices in $\Lambda^d$, so that $(\Lambda^d,E)$ is an oriented graph.
Let $\overleftarrow{E}$ be the set of edges with reverse orientation from $E$ and let $\overleftarrow{e} \in \overleftarrow{E}$ denote the reversal of $e\in E$.
A \textit{loop} is an ordered tuple $\ell=(\ell_1,\ldots,\ell_n)$, where $\ell_i\in E\sqcup \overleftarrow{E}$ and such that the terminal vertex of $\ell_i$ is the starting vertex of $\ell_{i+1}$ where indexes are modulo $n$.
We extend every $U\in G^E$ canonically to an element of $G^{E\sqcup \overleftarrow{E}}$ by setting $U(\overleftarrow{e}) = U(e)^{-1}$ for every $e\in E$.
Then, for a loop $\ell = (\ell_1,\ldots,\ell_n)$, we define the Wilson loop observable
\begin{equ}
W_\ell(U) = \Trace (U(\ell_1)\cdots U(\ell_n))\;.
\end{equ}
The following is a corollary of the proof of Theorem \ref{thm:carpet}.

\begin{corollary}\label{c.monotone}
Suppose $G=U(1)$ and let $\ell$ be a loop.
Then the expectation $\scal{W_\ell}_{\mu_{V,\bar\beta}} = \int W_\ell\mrd \mu_{V,\bar\beta}$ is a monotone function in the coupling constants $\bar \beta = (\beta_e)_{e\in E(\Lambda^d)}$.
That is, if $\bar \beta = (\beta_p)_{p\in P}$ and $\bar \gamma = (\gamma_p)_{p\in P}$ are such that $\beta_p \leq \gamma_p$ for all  $p$,
then $\scal{W_\ell}_{\mu_{V,\bar\beta}} \leq \scal{W_\ell}_{\mu_{V,\bar\gamma}}$.
\end{corollary} 
We give the proof of Corollary \ref{c.monotone} at the end of Section \ref{s.Proof}.

\begin{remark}\label{r.monotone}
It is nicely explained, for example in \cite{van2023duality}, that one can deduce such monotonies readily from the Ginibre inequality (without considering such geometric limits via cable/carpet graphs) for any interaction  $\theta \mapsto \exp(- U(\theta))$ which is such that $\theta \mapsto -U(\theta)$ is positive definite (i.e. all its Fourier coefficients $\{- \hat U (k)\}_{k\neq 0}$ are non-negative).
Unfortunately, it can be checked that 
\begin{align*}
\theta \mapsto  \log \left(\sum_{m \in \Z} \exp(-\frac \beta 2 (\theta + 2\pi m)^2) \right) 
\end{align*}
is not positive definite. Therefore, the present limiting procedure seems to be a necessary step in order to prove the monotony property stated in Corollary \ref{c.monotone}. 
\end{remark}

\subsection{Idea of proof.}
The proof handles the following two main difficulties which are addressed respectively in Sections \ref{s.commute} and \ref{sec:transition}:
\begin{enumerate}
\item {\em Restoring some commutativity.} When the gauge group $G$ is non-Abelian, we may fear (we did at least!) that we may not be able to split the interaction over plaquettes into microscopic interactions over many plaquettes. For example, it is well known that when $G$ is non-Abelian, we cannot express Wilson loop observables as a product over a spanning surface of microscopic loop observables. 
Also, spanning trees, which can be used to remove gauge freedom, help to reduce the problem when $d=2$ to a problem about random walks in a group but this is in general not the case when $d\geq 3$
(though see~\cite{chatterjee2016leading} where spanning trees are used in the form of axial gauge fixing for any $d\geq 2$).
It turns out we can still restore enough commutativity thanks to the straightforward Lemma \ref{l.easy} and ideas from planar (2D) gauge theory in the form of Lemma \ref{lem:reduction}.
\item  {\em Strong convergence of the heat-kernel.} The second difficulty which naturally arises is more analytical. When a plaquette $p$ is divided into $N^2$ small plaquettes each with inverse temperature $\beta_N = N^2 \beta$, we end up with the heat-kernel on the Lie Group $G$ after $T=N^2$ steps of a random walk $G$ with small random displacements. The convergence of this random-walk heat-kernel towards the heat-kernel of the limiting Brownian motion on the Lie-Group $G$ is well known (see for example \cite{Stroock_Varadhan_73,jorgensen1975central}).
But it is only established in the weaker notion of convergence in distribution.
In the present situation, each plaquette contributes such a heat-kernel and we need to control products of these heat-kernels.
Since such products  behave very poorly  under weak convergence, we crucially need to upgrade the notion of convergence of RW-heat-kernels to the $G$-valued Brownian motion heat-kernel.
This is the purpose of Section \ref{sec:transition} which is greatly inspired by \cite{Hebisch_Saloff_Coste_93} and~\cite{CS23}.
\end{enumerate}

\section{Assumption on actions}
\label{sec:assump}

Denote by $\mfg$ the Lie algebra of $G$ and equip $\mfg$
with the inner product $\scal{X,Y}=- \Trace(XY)$, which is $\Ad$-invariant.
We identify $\mfg$ with the space of left-invariant vector fields on $G$.
Let $D$ denote the dimension on $\mfg$ and let $T^1,\ldots, T^D$ be an orthonormal basis of $\mfg$.
Recall that a function $f\colon G\to\R$ is called a class function if $f(xyx^{-1})=f(y)$ for all $x,y\in G$
and is called symmetric if $f(x)=f(x^{-1})$ for all $x\in G$.

For $Q\colon G\to [0,\infty)$, let $Q^{\star k}$ denote the $k$-fold convolution of $Q$ with itself, i.e.
\begin{equ}[eq:Q_conv]
Q^{\star k}(x)=\int_{G^{k-1}} Q(x_1)Q(x_1^{-1} x_2) Q(x_2^{-1}x_3)\ldots Q(x_{k-1}^{-1} x) \mrd x_1\ldots \mrd x_{k-1}\;.
\end{equ}

\begin{assumption}\label{assump:actions}
Let $p_N\colon G\to [0,\infty)$ be a symmetric class function 
such that $\int_G p_N(x)\mrd x=1$. (Recall $\mrd x$ stands for the Haar measure on $G$). 
We make the following two assumptions.
\begin{enumerate}[label=(\alph*)]
\item \label{pt:conv} $p_N^{\star N^2}(x) \mrd x \to V_\beta(x)\mrd x$ in distribution as $N\to\infty$.

\item \label{pt:bounds} There exist $r,\theta,\Theta>0$ such that, for all $N\geq 1$, we can write $p_N(x)=Z_N^{-1}e^{-S_N(x)}$ with the property that
\begin{equ}
 \forall x\in B_{N^{-1} r}\eqdef \{x\in G\,:\, \rho(x,1_G)<N^{-1} r\}\;:\quad S_N(x) \leq \Theta N^2\rho(x,1_G)^2
\end{equ}
and
\begin{equ}
\forall x\in G\;:\quad S_N(x) \geq \theta N^2 \rho(x,1_G)^2\;,
\end{equ}
where $\rho$ is the geodesic distance on $G$ induced by $\scal{\cdot}$ and $1_G$ is the identity element of $G$.
\end{enumerate}
\end{assumption}

\begin{example}[Wilson action]\label{ex:Wilson_bounds}
Define $W_\beta\colon G\to (0,\infty)$ by
\begin{equ}
W_\beta(x) = Z^{-1}\exp(-\beta ({\Re\Trace (1_G-x) }))
\end{equ}
where $Z$ is such  that $\int_G W_\beta=1$.
Define $p_N = W_{N^2\beta}(x)$.
Then by \cite[Example 9.16]{CS23}, $p_N$ satisfies Assumption~\ref{assump:actions}\ref{pt:bounds}.
Assumption~\ref{assump:actions}\ref{pt:conv}, i.e. that $p_N^{\star N^2}(x)\mrd x \to V_\beta(x)\mrd x$ in distribution, follows from Proposition~\ref{prop:Wilson_conv} below.
\end{example}

\begin{example}[Manton action]\label{ex:Manton_bounds}
Define $M_\beta,\colon G\to (0,\infty)$ by
\begin{equ}
M_\beta(x) = Z^{-1}\exp(-\beta \rho(x,1)^2)
\end{equ}
where $Z$ is such  that $\int_G M_\beta=1$.
Then $p_N \eqdef M_{N^2\beta}(x)$ trivially satisfies Assumption~\ref{assump:actions}\ref{pt:bounds}.
Assumption~\ref{assump:actions}\ref{pt:conv} follows from Proposition~\ref{prop:Manton_conv} below.
\end{example}

We next give a general way to verify the convergence $\P_N^{\star N^2} \to V_\beta(x)\mrd x$ in distribution.
Let $\xi^1,\ldots, \xi^D \in \CC^\infty(G,\R)$ be local exponential coordinates of the first kind associated to the basis $T^a$, i.e. the map $\xi\colon G\to \mfg$ defined by
\begin{equ}
\xi (x) \eqdef \sum_{a=1}^D \xi^a(x) T^a
\end{equ}
satisfies $e^{\xi(x)}=x$ for all $x$ in a neighbourhood of $1_G$.
We assume without loss of generality that $\xi$ is odd, i.e. $\xi(x)=-\xi(x^{-1})$ (otherwise we can replace $\xi^a(x)$ by $\frac12 \xi^a(x)-\frac12\xi^a(x^{-1})$).

Suppose $\P_N$ are probability measures on $G$,
let $\E_N$ denote the associated expectations, and define
\begin{equ}
B_N = \E_N[\xi]\;,\quad
A^{a,b}_N = \E_N[\xi^a\xi^b]\;,\quad 1\leq a,b \leq D\;.
\end{equ}

Generalising \eqref{eq:Q_conv}, if $\P$ is a probability measure on $G$, we let $\P^{\star k}$ denote its $k$-fold convolution,
which is just the law of the $k$-th step $X_k$ of the random walk on $G$ with $X_0=1_G$ and whose increments $X_{i-1}^{-1}X_{i}$ are i.i.d. and distributed by $\P$.

\begin{lemma}\label{lem:RW}
Suppose $\lim_{N\to \infty} \P_N(C) = 0$ for every closed set $C\subset G$ such that $1_G\notin C$.
Suppose further that $N B_N\to 0$ and $N A^{a,b}_N\to A^{a,b}$ as $N\to\infty$.

Then $\P_N^{\star N} \to \P$ in distribution where $\P(\mrd x) = e^{\CL}(x)\mrd x$ is the law at time $1$ of the $G$-valued diffusion with generator $\CL = \frac12\sum_{a,b=1}^D A^{a,b} T^a T^b$
(i.e. $e^{t\CL}$ is the heat-kernel of $\CL$).
\end{lemma}

\begin{remark}
This result does not require orthonormality of $T^a$ or compactness of $G$;
it is a special case of a general ``functional'' central limit theorem for walks with independent increments that applies to any connected Lie group, see~\cite{Feinsilver78,Stroock_Varadhan_73,jorgensen1975central}
or~\cite[Sec.~2.2]{Chevyrev18}, the notation of which we follow.
It also follows from Wehn's central limit theorem~\cite{Wehn_thesis,Wehn62,Grenander63}, see also the survey~\cite{Breuillard_07}.
\end{remark}

\begin{proposition}\label{prop:Wilson_conv}
For the Wilson action from Example \ref{ex:Wilson_bounds}, we have
\[
W_{N\beta}^{\star N}(x) \mrd x \to V_\beta(x)\mrd x
\]
in distribution as $N\to\infty$.
\end{proposition}

\begin{proof}
Let $\P_N(\mrd x) = W_{N\beta}(x)\mrd x$ and follow notation as above.
By symmetry and the assumption that $\xi$ is odd, note that $B_N=0$.

There exists a connected domain $F\subset \mfg$ such that $\exp\colon F\to G$ is a bijection (e.g. $F=\{X\in\mfg\,:\, \sigma(X)\subset (-i\pi,i\pi]\}$ where $\sigma(X)$ is the spectrum of $X$).
Equip $\mfg$ with the Lebesgue measure $\mrd X$ that assigns
unit volume to the unit cube with respect to $\scal{\cdot,\cdot}$
(i.e. the map $\mfg\ni \sum_a X^aT^a\mapsto (X^1,\ldots,X^D)\in \R^D$ is a measure preserving isometry for the standard Lebesgue measure on $\R^D$).
Let $J\colon F\to [0,\infty)$ denote the Jacobian of $\exp$, i.e. the pushforward of $J(X)\mrd X$ via $\exp$ is the Haar measure on $G$.
Then
\begin{equ}[eq:xi2_Wilson]
N \E_N [\xi\otimes \xi] = Z_N^{-1} \int_F N \xi(e^X)^{\otimes 2} e^{-N\beta (\frac12\|X\|^2 + R_N(X))} J(X)\mrd X\;,
\end{equ}
where $R_N(X) = \Trace(\cosh(X)-1-X^2/2!) = O(X^4)$
and where we recall that $\|X\|^2 = -\Trace(X^2)$ (see, e.g.~\cite[Sec.~2.3.1]{CS23}).
By the derivative of the exponential map formula~\cite[Sec.~1.2, Thm.~5]{RossmannBook}, note that $J$ is smooth and bounded from above on $F$ and strictly positive in a neighbourhood of $0$.

We readily see that, for every closed $C\subset G$ with $1_G\notin C$, there exists $K>0$ such that $\P_N(C)\leq K e^{-N/K}$,
so the first condition of Lemma~\ref{lem:RW} is satisfied.
Furthermore, by a change of variable,
\begin{multline}\label{eq:int}
N\E_N [\xi\otimes \xi]
=
J(0)Z_N^{-1}N^{-D/2}
\int_{\sqrt N F} N \xi(e^{Y/\sqrt N})^{\otimes 2}
e^{-\beta (\frac12\|Y\|^2 + N R_N(Y/\sqrt N))}
\\
\frac{J(Y/\sqrt N)}{J(0)} \mrd Y\;.
\end{multline}
Observe that $J(X)/J(0)$ is bounded from above and, for $X$ close to $0$, is equal to $1+\ell(X) + O(X^2)$ for some linear map $\ell\colon\mfg\to \R$.
Since $\xi(e^X)=X$ for $X$ close to $0$,
it follows that the integrand in~\eqref{eq:int} converges pointwise on $\mfg$ to $
Y^{\otimes 2} e^{-\beta\|Y\|^2}
$
and is uniformly bounded,
hence, by dominated convergence, the integral converges to $\int_\mfg Y^{\otimes 2}e^{-\frac12 \beta\|Y\|^2} \mrd Y$.
By the same reasoning (substituting $N\xi^{\otimes 2}$ by $1$)
$J(0)^{-1}Z_N N^{D/2}\to \int_\mfg e^{-\frac12 \beta\|Y\|^2} \mrd Y = (2\pi \beta^{-1})^{D/2}$.
In conclusion,
\begin{equ}
N\E_N[\xi\otimes \xi]
\to (2\pi \beta^{-1})^{-D/2} \int_\mfg Y^{\otimes 2} e^{-\frac12 \beta\|Y\|^2}\mrd Y
= \beta^{-1} \sum_a T^a\otimes T^a\;.
\end{equ}
Hence $N A_N^{a,b} \to A^{a,b} \eqdef \beta^{-1}\delta_{ab}$,
and the conclusion follows from Lemma \ref{lem:RW}.
\end{proof}

\begin{proposition}\label{prop:Manton_conv}
For the Manton action from Example \ref{ex:Manton_bounds}, we have 
\[
M_{N\beta}^{\star N}(x) \mrd x \to V_\beta(x)\mrd x
\]
 in distribution as $N\to\infty$.
\end{proposition}

\begin{proof}
Identical to the proof of Proposition~\ref{prop:Wilson_conv} once we use the fact that $\rho(e^X,1) = \|X\|$ for $X$ close to $0$ and replace $R_N$ in~\eqref{eq:xi2_Wilson} by $0$ for $X$ close to $0$.
\end{proof}

\section{Partition function of planar lattice gauge theories}\label{s.commute}

In this section we state and prove Lemma~\ref{lem:reduction} concerning the partition function of gauge theories on planar graphs.
The result is likely known, but we were unable to find a statement in the literature that is applicable in our setting.
See \cite[Sec.~1.6]{Levy03}, \cite[Thm.~6.4]{Driver89},
and \cite[Sec.~2.3]{Witten91} for somewhat similar considerations as well as the proof of \cite[Thm.~5.3]{Chevyrev19YM} for a special case.

Let $\G$ be an oriented connected finite multigraph in the plane with edge set $\bar E$ and face set $\bar F$.
Let $f_\infty\in \bar F$ denote the outer (infinite) face and define $F = \bar F\setminus \{f_\infty\}$,
elements of which are called \textit{internal faces}.
Let $E$ denote the set of \textit{internal edges}, i.e. edges not bordering $f_\infty$
and denote $\d E = \bar E\setminus E$.
For an edge $e\in \bar E$, we denote by $\overleftarrow{e}$ the edge with opposite orientation (which is necessarily not in $\bar E$)
and we denote $\overleftarrow{E} = \{\overleftarrow{e} \,:\, e\in \bar E\}$.

For a function $U\colon \bar E \to G$, we extend its definition to $\overleftarrow{e}$ for every $e\in \bar E$ by setting $U(\overleftarrow{e})=U(e)^{-1}$.
For every $f\in \bar F$, let $U(f)$ denote the conjugacy class in $G$ of the product $U(e_1)\cdots U(e_n)$ where $e_i$ are the edges bordering $f$ ordered and oriented in the clockwise direction
and every $e_i$ is either in $\bar E$ or $\overleftarrow{E}$ 
(note that $U(f)$ is well-defined since it is independent of the choice of starting edge $e_1$),
see Figure~\ref{fig:planar_graph}.

\begin{figure}
\centering
\begin{tikzpicture}[scale=1.3,vertex/.style={circle, draw, fill=black, inner sep=0pt, minimum size=3mm}]
    \node[vertex] (A) at (0,0) {};
    \node[vertex] (B) at (2,0) {};
    \node[vertex] (C) at (4,0) {};
    \node[vertex] (D) at (2,-2) {};
    \node[vertex] (E) at (0,-2) {};
    \node[vertex] (F) at (0.7,-0.7) {}; 

    \draw[-{Latex[length=3mm]}] (A) -- (B) node[midway, above] {$e_1$}; 
    \draw[{Latex[length=3mm]}-] (B) -- (C) node[midway, above] {$e_2$}; 
    \draw[-{Latex[length=3mm]}] (C) -- (D) node[midway, right] {$e_3$}; 
    \draw[-{Latex[length=3mm]}] (D) -- (E) node[midway, above] {$e_4$}; 
    \draw[-{Latex[length=3mm]}] (E) -- (A) node[midway, left] {$e_5$}; 
    \draw[{Latex[length=3mm]}-] (B) to[out=-45, in=135] node[pos=0.5, sloped=0, left] {$e_8$} (D); 
    \draw[-{Latex[length=3mm]}] (D) to[out=-60, in=-120] node[pos=0.5, below] {$e_9$} (E); 
    \draw[{Latex[length=3mm]}-] (A) -- (F) node[pos=0.8, above, xshift=4] {$e_{10}$}; 

    \node at (1,-1.2) {$f_1$};
    \node at (1, -2.3) {$f_2$};
    \node at (2.7, -0.6) {$f_3$};
	\node at (3.5, -1.5) {$f_\infty$};
\end{tikzpicture}

\caption{Example of $\G$.
Boundary edges are $e_1,e_2,e_3,e_9,e_5$.
Denoting by $U_i = U(e_i)$, we have that $U(f_1)$ is the conjugacy class of $ U_8^{-1}U_4U_5U_{10}^{-1}U_{10}U_1=U_8^{-1}U_4U_5U_1$,
 $U(f_2)$ is the conjugacy class of $U_4^{-1}U_9$,
 $U(f_3)$ is the conjugacy class of $U_3U_8 U_2^{-1}$,
and
$U(f_\infty)$ is the conjugacy class of $U_1^{-1}U_5^{-1}U_9^{-1}U_3^{-1}U_2$.
The edge $e_{10}$ borders only $f_1$ and $U(f_1)$ does not depend on $U_{10}$ as per Remark~\ref{rem:one_face}.}
\label{fig:planar_graph}
\end{figure}

\begin{remark}\label{rem:one_face}
An edge $e$ may border only one face $f$;
in this case $e$ appears twice consecutively in the sequence $e_1,\ldots,e_n$ with
opposite orientation,
i.e. $e=e_i$ and either $\overleftarrow{e}=e_{i+1}$ or $\overleftarrow{e}=e_{i-1}$ for some $i=1,\ldots,n$ (indexes are mod $n$).
Since $U(e)=U(\overleftarrow{e})^{-1}$,
$U(e_1)\cdots U(e_n)$ does not depend on $U(e)$, see Figure~\ref{fig:planar_graph}.
\end{remark}

Suppose we are given ``boundary conditions'' $\d  U \colon \d E \to G$.
Note that $\d U(f_\infty)$ is well-defined from this data.

For any $U\in G^E$, we extend it to a map $U\in G^{\bar E}$, which we denote by the same symbol, by setting $U(e)=\d U(e)$ for all $e\in\d E$.
In particular, for all $U\in G^E$, the conjugacy class $U(f)$ is well-defined for every $f\in \bar F$.

\smallskip

Recall the convolution of two functions $p,q\colon G\to \R$ is
\begin{equ}
(p\star q)(x) = \int_G p(y)q(y^{-1}x)\mrd y
\end{equ}
whenever the integral exists. We will need the following easy but key Lemma.
\begin{lemma}\label{l.easy}
We have $p\star q = q\star p$ for any $p,q\colon G\to \R$ whenever one of $p$ or $q$ is a class function.
\end{lemma}
\begin{proof}
If $p$ is a class function, then by a change of variable $z=y^{-1}x$,
\begin{equs}
(p\star q)(x) &= \int_G p(y)q(y^{-1}x)\mrd y= \int_G p(xz^{-1})q(z)\mrd z
\\
&=
\int_G p(z^{-1}x)q(z)\mrd z
= (q\star p)(x)\;.
\end{equs}
A similar argument applies if $q$ is a class function.
\end{proof}

\begin{lemma}\label{lem:reduction}
Suppose that $F$ is non-empty and for every $f\in F$, let $p_f\colon G\to\R$ be a bounded measurable class function.
Then for any prescribed boundary condition  $\d  U \colon \d E \to G$,
\begin{equ}[eq:reduction]
\int_{G^{E}} \prod_{f\in F} p_f(U(f)) \prod_{e\in E}\mrd U_e = p_F (\d U (f_\infty)^{-1})
\end{equ}
where $p_F=p_{f_1}\star \ldots \star p_{f_n}$ and $f_1,\ldots,f_n$ is an arbitrary enumeration of the faces in $F$.
If $E$ is empty, in which case $F=\{f_1\}$ is a singleton, the left-hand side of~\eqref{eq:reduction} is $p_{f_1} (\d U (f_\infty)^{-1})$ by convention.
\end{lemma}

\begin{proof}
We proceed by induction on the number of elements in $F$.
The base case is when
$F=\{f_1\}$ is a singleton.
In this case, if $E$ is empty, the claim is true by assumption.
Otherwise, $E$ consists only of internal edges that border $f_1$
and $U(f)$ does not depend on $U_e$ with $e\in E$ by Remark~\ref{rem:one_face}.
Furthermore $U(f) = \d U(f_\infty)^{-1}$ and $\int_{G^E}  \prod_{e\in E}\mrd U_e =1$, which concludes the proof of the base case.

Suppose now the claim is true for all graphs with $n$ internal faces with $n\geq 1$ and consider $\G$ such that $|F|=n+1$.
There necessarily exists an internal edge $e\in E$ which borders two different internal faces, say $f,f'\in F$.
We suppose $e$ appears in the clockwise direction in $f$ and counterclockwise in $f'$.
Then the left-hand side of~\eqref{eq:reduction} takes the form
\begin{equ}[eq:one_edge]
\int_{G^{E\setminus \{e\}} }\prod_{k\neq e} \mrd U_k \int_G \mrd U_e p_f(U_e U(p)) p_{f'}(U(p')U_e^{-1} ) \prod_{g\neq f,f'} p_g(U(g))\;,
\end{equ}
where $U(p)$ and $U(p')$ are ordered products among remaining edges bordering $f$ and $f'$ respectively.
By changing variable $U_e \mapsto V_e = U_e U(p)$, so that $U_e^{-1}=U(p) V_e^{-1}$,
taking the inner integral and using that $p_{f'}$ is a class function, \eqref{eq:one_edge} equals
\begin{equ}
\int_{G^{E\setminus \{e\}}} \prod_{k} \mrd U_k (p_f\star p_{f'})(U(p')U(p)) \prod_{g\neq f,f'} p_g(U(g))\;.
\end{equ}
This is precisely the left-hand side of \eqref{eq:reduction} for the graph formed by removing the edge $e$,
merging the faces $f,f'$ into one face $f''$,
and taking $p_{f''} = p_{f}\star p_{f'}$.
The result follows by induction since we have reduced the number of faces by one.
\end{proof}

\section{Proof of the main Theorem}\label{s.Proof}


In this Section, we conclude the proof of Theorem~\ref{thm:carpet}.
We write $X\lesssim Y$ if $X\leq CY$ for a proportionality constant $C>0$, the dependencies of which are either made explicit or are clear from the context.
We write $X\asymp Y$ if $X\lesssim Y$ and $Y\lesssim X$.

For a function $Q\colon G\to \R$ and $\e>0$, define $\nabla_\e Q\colon G\to [0,\infty)$ by
\begin{equ}
\nabla_\e Q(x) = \e^{-1}\sup_{y\in B_\e} |Q(x)-Q(xy)|\;,
\end{equ}
where we recall $B_\e = \{y\in G\,:\, \rho(1_G,y)<\e\}$ and $\rho$ is the geodesic distance on $G$.

\begin{lemma}\label{lem:p_grad_bound}
Suppose $p_N$ satisfies Assumption~\ref{assump:actions}.  
Then $\sup_N \|\nabla_{N^{-1}}p^{\star N^2}_{N}\|_\infty < \infty$.
\end{lemma}

\begin{proof}
Since $\|\nabla_\e (f\star p)\|_\infty \leq \|\nabla_\e f \|_\infty$ whenever $\int p=1$,
it suffices to consider $N^2=2^m$ for some $m\geq 1$.
By the proof of \cite[Prop.~9.15]{CS23},
the conditions of Proposition~\ref{prop:gradient_bound} below (which are identical to the conditions of \cite[Prop.~C.5]{CS23})
are satisfied with $\e=N^{-1}$ and $p=p_N$.
Proposition~\ref{prop:gradient_bound} therefore implies
$\|\nabla_{N^{-1}} p^{\star N^2}_{N} \|_\infty  \lesssim 1$.
(Note that we did not use Assumption~\ref{assump:actions}\ref{pt:conv} here.)
\end{proof}

The reason the control on the gradient is important is that it allows us to improve convergence in distribution to uniform convergence:

\begin{lemma}\label{lem:Q_N_conv}
Let $Q_N,Q\colon G\to \R$ such that $Q_N \to  Q$ is the sense of (Schwartz) distributions and $\sup_N \|\nabla_{\e(N)} Q_N\|_\infty <\infty$ for some sequence $\e(N)>0$ with $\e(N)\to 0$.
Then $Q_N\to Q$ uniformly.
\end{lemma}

\begin{proof}
The proof is easy but we include it for completeness.
Let $\chi_N$ be a mollifier at scale $\e$ (we drop reference to $N$ in $\eps$).
Then $Q_N\star \chi_N$ is smooth and for $\rho(x,y)<\e$
\begin{equs}
|Q_N\star \chi_N(x)-Q_N\star \chi_N(y)|
&=\Big|\int_G \{\chi_N(z^{-1}x)-\chi_N(z^{-1}y)\}Q_N(z)\mrd z\Big|
\\
&\lesssim \rho(x,y)
\end{equs}
where we used that the integration domain has volume $\e^D$ and $\|D\chi_N\|_\infty\lesssim \e^{-D-1}$ and that the integral of $\chi_N(z^{-1}x)-\chi_N(z^{-1}y)$ is zero and has support in a ball radius $\lesssim \e$ on which $Q_N$ oscillates by at most $\lesssim \e$.  We may now relax the condition $\rho(x,y) < \eps$ using the triangle inequality along geodesics.

It follows that $Q_N\star \chi_N$ is uniformly equicontinuous and uniformly bounded
and thus relatively compact in $\CC(G)$ by Arzel\`a--Ascoli.
On the other hand,  $\|Q_N -  Q_N\star\chi_N\|_\infty \leq \eps\|\nabla_{\e}Q_N\|_\infty$ which convergences to $0$ as $N\to\infty$.
It follows that $Q_N$ converges along subsequences and has unique limit point $Q$.
\end{proof}

\begin{proof}[of Theorem \ref{thm:carpet}]
Consider the graph $\tilde \Lambda^d_N$ formed by dissecting every edge of $\Lambda^d \subset \Z^d$ into $N$ edges of length $\eps = N^{-1}$. (In particular, we have the inclusions $\Lambda^d_N \subset \tilde \Lambda^d_N \subset \Lambda^d$). 
\smallskip

Let $\tilde E_N$ denote its edge set, which we treat as a subset of $E_N$.
As in Remark~\ref{rem:projection},
there is a canonical projection $\tilde\pi_N \colon G^{\tilde E_N} \to G^E$
given by $\tilde \pi_N U(e) = U(e_1)\ldots U(e_N)$ where $e_1,\ldots,e_N$ are the edges appearing in $e$.
Then, for any bounded measurable function $f\colon G^{E}\to \R$,
\begin{equs}
(\pi_N)_*\mu_{p_N,N}(f) &= Z^{-1}\int_{G^{E_N}} f(\pi_N U)\prod_{p\in P_N} p_N(U_p) \mrd U
\\
&=
\tilde Z^{-1} \int_{G^{\tilde E_N}} f(\tilde \pi_N U)\prod_{p\in P} p_{N}^{\star N^2}((\tilde \pi_N   U)_p)  \prod_{e\in \tilde E_n }\mrd  U_e
\\
&=
Z_{p_{N}^{\star N^2}}^{-1} \int_{G^{E}} f(U)\prod_{p\in P} p_{N}^{\star N^2}(U_p)  \prod_{e\in E} \mrd U_e
= \mu_{p_{N}^{\star N^2}}(f)\;,
\end{equs}
where $\tilde Z$ is the normalisation constant so that the second line equals $1$ for $f\equiv 1$.
Above,
the second equality follows by applying Lemma~\ref{lem:reduction} successively to each plaquette $p\in P$ (indeed each of these plaquettes  correspond to an $N*N$ grid and is thus a planar graph so that Lemma~\ref{lem:reduction} applies). As such Lemma~\ref{lem:reduction} allows us to integrate out the Haar measures carried by all the edges in $E_N \setminus \tilde E_N$.  Note also that $ f(\pi_N U) =  f(\tilde \pi_N  U)$  since $f(\pi_N U)$ does not depend on any edges in $E_N\setminus \tilde E_N$. The third equality
follows from a change of variable that allows us to integrate out the $N-1$ redundant degrees of freedom on each edge in $E$
(in particular $\tilde Z=Z_{p_{N}^{\star N^2}}$).

By Assumption \ref{assump:actions},
$p_{N}^{\star N^2}$ converges to $V_\beta$ in the sense of distributions.
By Lemmas \ref{lem:p_grad_bound}-\ref{lem:Q_N_conv},
this convergence is moreover uniform.
It follows that the density of $\mu_{p_{N}^{\star N^2}}$ converges (uniformly) to the density of $\mu_{V_\beta}$, which completes the proof.
\end{proof}

\begin{proof}[of Corollary \ref{c.monotone}]
Similar to Remark \ref{rem:projection}, every loop $\ell$ in the graph $E$ canonically defines a loop $\ell_N$ on $E_N$ with $N$ times as many edges.
Then, for any $U\in G^{E_N}$, one has $W_{\ell_N}(U) = W_\ell(\pi_N U)$ by construction,
so in particular $\scal{W_{\ell_N}}_\mu = \scal{W_\ell}_{(\pi_N)_*\mu}$ for any measure $\mu$ on $G^{E_N}$.

We consider now the Wilson($=$XY) lattice gauge theory $\mu_N = \mu_{W_{N^2\bar\beta},N}$ on $E_N$ as a slight generalisation of Example \ref{ex:Wilson_bounds} where we allow different couplings $\bar\beta = (\beta_p)_{p\in P(\Lambda^d)}$ as in \eqref{eq:Villain_betas}.
(The coupling $\beta_p$ assigned to a microscopic plaquette $p\in P_N$ is given by $\beta_q$ for the unique $q\in P$ such that $p$ is part of the tiling of $q$, recall Figure \ref{f.Carpet}.)

Monotony of Wilson loops in $\bar\beta$ holds for the Wilson theory $\mu_N = \mu_{W_{N^2\bar\beta},N}$ thanks to the Ginibre inequality~\cite{ginibre1970general} (in fact, it holds for the Wilson theory on any graph).
We thus obtain monotony of $\scal{W_{\ell_N}}_\mu=\scal{W_\ell}_{(\pi_N)_*\mu_N}$ in $\bar\beta$.
However, since the Wilson action satisfies Assumption \ref{assump:actions} due to Example \ref{ex:Wilson_bounds},
an obvious modification of the proof of Theorem \ref{thm:carpet},
implies that $(\pi_N)_*\mu_N \to \mu_{V,\bar\beta}$ in total variation distance.
Consequently $\lim_{N\to\infty}\scal{W_\ell}_{(\pi_N)_*\mu_N} = \scal{W_\ell}_{V,\bar\beta}$
and thus $\scal{W_\ell}_{V,\bar\beta}$ is also monotone in $\bar\beta$.
\end{proof}

\section{Transition functions of random walks on groups}
\label{sec:transition}

Let $G$ be a unimodular locally compact group equipped with a Haar measure, which we denote by $|A| = \int_A \mrd x$ for Borel measurable $A\subset G$.
Suppose $p \colon G\to [0,\infty)$ is a symmetric transition function, i.e. $\int p(x)\mrd x = 1$ and $p(x)=p(x^{-1})$.
We let $1_G$ denote the identity element of $G$ as before.
We derive in this section gradient estimates on the convolution power $p^{(m)} \eqdef p^{\star m}$.
Our results are similar to those of~\cite[Appendix~C]{CS23}, which in turn are inspired by~\cite{Hebisch_Saloff_Coste_93}.
We stress, however, that the estimates do not follow from~\cite{CS23} since gradient estimates are not considered therein.
We have made this section as self-contained as possible, only referencing~\cite{Hebisch_Saloff_Coste_93,CS23} for statements that require no change.

We fix in this section a measurable set $\Omega\subset G$ such that $1_G\in \Omega$ and that is symmetric, i.e. $x^{-1}\in\Omega$ for all $x\in\Omega$.
Define $\rho(x) = \inf\{n\geq 1\,:\,x\in \Omega^n\}$
and for measurable $f\colon G\to \R$ define
\begin{equ}
\nabla f(x) = \sup_{y\in\Omega} |f(x)-f(xy)|\;.
\end{equ}

\begin{theorem}\label{thm:weak_Harnack}
Suppose $c_0\eqdef \inf_{x\in\Omega^2}p^{(2)}(x) >0$ and let $B=c_0|\Omega|$.
Then for all $n,m\geq 1$
\begin{equ}[eq:grad_p_bound]
\|\nabla p^{(2n+m)}\|_\infty \leq 2 (m B)^{-1/2} p^{(2n)}(1_G)
\end{equ}
and
\begin{equ}[eq:x_e_bound]
|p^{(2n+m)}(x)-p^{(2n+m)}(1_G)| \leq 2 (m B)^{-1/2}\rho(x) p^{(2n)}(1_G)\;.
\end{equ}
\end{theorem}

For the proof of Theorem~\ref{thm:weak_Harnack}, we require the following two lemmas, the first of which is a quantitative version of~\cite[Lem.~3.2]{Hebisch_Saloff_Coste_93}.
\begin{lemma}\label{lem:spectral}
Let $K$ be a symmetric Markov operator. Then for all integers $m,n\geq 1$
\begin{equ}
\|(\id -K^{2n})^{1/2}K^{m}f\|_{L^2}^2 \leq \frac{n}{2m}\|f\|^2_{L^2}\;.
\end{equ}
\end{lemma}

\begin{proof}
Let $K=\int_{-1}^1 \lambda E(\mrd \lambda)$ be the spectral decomposition of $K$.
Then
\begin{equ}
\|(\id-K^{2n})^{1/2} K^m f\|^2_{L^2} = \int_{-1}^1(\lambda^{2m} - \lambda^{2n+2m}) \scal{E(\mrd \lambda)f,f} \leq \frac{n}{2m}\|f\|^2_{L^2}\;,
\end{equ}
where we used that $\lambda^{2m} - \lambda^{2n+2m} \leq n/(2m)$ for all $\lambda\in[-1,1]$.
\end{proof}

For measurable $f\colon G\to \R$, define
\begin{equ}
\nabla_{2} f(x) = \Big(\int_{\Omega^2}
|f(x)-f(xy)|^2 \mrd y\Big)^{1/2}\;.
\end{equ}

\begin{lemma}\label{lem:gradient_bound}
One has $\nabla f(x) \leq 2 |\Omega|^{-1/2}\sup_{y\in\Omega} \nabla_{2} f(xy)$ for all
$f\in L^1\cap L^\infty$.
\end{lemma}

\begin{proof}
This is~\cite[Lem.~3.1]{Hebisch_Saloff_Coste_93}.
\end{proof}

\begin{proof}[of Theorem~\ref{thm:weak_Harnack}]
Clearly $|p^{(2n+m)}(x)-p^{(2n+m)}(1_G)| \leq \rho(x)\|\nabla p^{(2n+m)}\|_\infty$, so \eqref{eq:x_e_bound} follows from \eqref{eq:grad_p_bound}.
We now prove the latter.
By Lemma~\ref{lem:gradient_bound},
\begin{equ}
\|\nabla p^{(2n+m)}\|_\infty
\leq 2|\Omega|^{-1/2}\|\nabla_{2}p^{(2n+m)}\|_\infty\;,
\end{equ}
We next use $p^{(2n+m)} = p^{(n+m)}\star p^{(n)}$, Minkowski's integral inequality, and Cauchy-Schwarz to obtain
\begin{equs}
|\nabla_{2} p^{(2n+m)}(x)|
&= \Big|\int_{\Omega^2} \Big|\int_G  \{p^{(n+m)}(y^{-1}x) - p^{(n+m)}(y^{-1}xz)
\}
p^{(n)}(y)\mrd y\Big|^2 \mrd z\Big|^{1/2}
\\
&\leq \int_{G} \Big|\int_{\Omega^2}  |p^{(n+m)}(y^{-1}x) - p^{(n+m)}(y^{-1}xz) |^2 \mrd z \Big|^{1/2} p^{(n)}(y)\mrd y
\\
&= \int_{G} |\nabla_{2} p^{(n+m)}(y^{-1}x)| p^{(n)}(y) \mrd y
\\
&\leq \|p^{(n)}\|_{L^2}\|\nabla_{2} p^{(n+m)}\|_{L^2}\;.
\end{equs}
Note that $\|p^{(n)}\|_{L^2}^2 = p^{(2n)}(1_G)$ by symmetry of $p$.
Furthermore, by the assumption $p^{(2)}\geq c_0$ on $\Omega^2$, denoting by $P$ the convolution operator $Pf = p\star f$,
\begin{equs}
\|\nabla_{2} p^{(n+m)}\|_{L^2}^2
&= \int_G \int_{\Omega^2} |p^{(n+m)}(x)
-p^{(n+m)}(xy)|^2 \mrd y \mrd x
\\
&\leq c_0^{-1} \int_G\int_G |p^{(n+m)}(x)
-p^{(n+m)}(xy)|^2 p^{(2)}(y) \mrd y \mrd x
\\
&= 2c_0^{-1}\|(\id - P^2)^{1/2}P^m p^{(n)}\|^2_{L^2}
\\
&\leq 2c_0^{-1}\frac{1}{2m}\|p^{(n)}\|^2_{L^2} =  \frac{1}{c_0 m}p^{(2n)}(1_G)\;,
\end{equs}
where we used that $\int_{G^2}p^{(n+m)}(xy)^2 p^{(2)}(y)\mrd y\mrd x = \|p^{(n+m)}\|^2_{L^2}$ in the third line and Lemma~\ref{lem:spectral} in the fourth line.
\end{proof}

\begin{theorem}\label{thm:pe_bound}
Suppose there exist $a\in [0,\infty]$, $C,D\geq 0$ such that, for all $n \geq 1$,
\begin{equ}\label{eq:doubling}
|\Omega^n| \geq C\min\{a, |\Omega|n^D\}\;.
\end{equ}
With the notation and assumptions of Theorem \ref{thm:weak_Harnack}, one has for all $m\geq1$
\begin{equ}\label{eq:induction_p_e}
p^{(2^m)}(1_G) \lesssim \max\{2^{-Dm/2}B^{-D/2}c_1,
C^{-1}a^{-1},
C^{-1}2^{-Dm/2} |\Omega|^{-1}B^{-D/2}\}\;,
\end{equ}
where 
$c_1=p^{(2)}(1_G)$ and
the proportionality constant depends only on $D$.
\end{theorem}

\begin{proof}
This is~\cite[Thm.~C.2]{CS23}.
\end{proof}
The following consequence of Theorems~\ref{thm:weak_Harnack} and~\ref{thm:pe_bound} is much more specific but is how we use the above results.
\begin{proposition}\label{prop:gradient_bound}
Suppose that
\begin{itemize}
\item  $G$ is a compact connected Lie group of dimension $D$,
\item $\Omega$ is a ball centred at $1_G$ of radius $\eps \in (0, \eps_0)$ for a geodesic distance on $G$, where $\eps_0>0$ is sufficiently small, and
\item $p^{(2)}(1_G) \asymp \inf_{x\in \Omega^2} p^{(2)}(x) \asymp \eps^{-D}$ uniformly in $\eps \in (0,\eps_0)$.
\end{itemize}
Then, uniformly in $\e\in (0,\eps_0)$ and $m\geq 0$,
\begin{equ}
\|\nabla p^{(2^m)}\|_\infty \lesssim 2^{-m/2}\max\{1,2^{-Dm/2}\eps^{-D}\}\;.
\end{equ}
\end{proposition}

\begin{proof}
We have $|\Omega|\asymp \eps^D$
and~\eqref{eq:doubling} holds with $a=1$ and $C\asymp 1$.
Let $c_1=p^{(2)}(1_G)$ and $c_0 = \inf_{x\in \Omega^2} p^{(2)}(x)$.
Then $B\eqdef c_0|\Omega| \asymp 1$
and the right-hand side of~\eqref{eq:induction_p_e} is bounded above by a multiple of $\max\{1, 2^{-Dm/2}\eps^{-D}\}$ for all $m\geq 1$.
The conclusion now follows from Theorem \ref{thm:pe_bound} and \eqref{eq:grad_p_bound}.
\end{proof}

\noindent
\textbf{Acknowledgments.}
The authors are grateful to the anonymous referees for their careful reading of the manuscript and helpful comments.
They also wish to thank Fabrice Baudoin, Diederik van Engelenburg and Avelio Sep\'ulveda for useful discussions as well as the University of Geneva (Unige) where this work has been initiated in Spring 2023. 
I.C. acknowledges support from the EPSRC via the New Investigator Award EP/X015688/1.
C.G. acknowledges support from the Institut Universitaire de France (IUF), the ERC grant VORTEX 101043450 and the French ANR grant ANR-21-CE40-0003.

\smallskip

\noindent
\textbf{Data availability statement.} The authors confirm that the manuscript has no associated data.

\smallskip
\noindent\textbf{Statement on competing interests.} The authors declare they have no competing interests. 

\appendix

\endappendix

\bibliographystyle{./Martin}
\bibliography{./refs}

\end{document}